\documentclass[reqno]{amsart}
\usepackage{amsmath,amsthm,amssymb,cite}

\makeatletter
    
    \@addtoreset{equation}{section}
  \makeatother

\newtheorem{definition}{Definition}[section]
\newtheorem{proposition}[definition]{Proposition}
\newtheorem{theorem}[definition]{Theorem}
\newtheorem{lemma}[definition]{Lemma}
\newtheorem{corollary}[definition]{Corollary}
\newtheorem{remark}[definition]{Remark}

\pagebreak

\title[Semilinear damped wave equations]{Small data global existence for the semilinear wave equation
with space-time dependent damping}
\author{Yuta WAKASUGI}
\address{Department of Mathematics, Graduate School of Science, Osaka University, Osaka, Toyonaka, 560-0043, Japan}
\email{y-wakasugi@cr.math.sci.osaka-u.ac.jp}
\begin{document}
\begin{abstract}
In this paper we consider the critical exponent problem for the semilinear wave equation
with space-time dependent damping. 
When the damping is effective, it is expected that the critical exponent agrees with
that of only space dependent coefficient case.
We shall prove that there exists a unique global solution for small data if the power of nonlinearity
is larger than the expected exponent.
Moreover, we do not assume that the data are compactly supported.
However, it is still open whether there exists a blow-up solution
if the power of nonlinearity is smaller than the expected exponent.
\end{abstract}

\maketitle
\section{Introduction}
We consider the Cauchy problem for the semilinear damped wave equation
\begin{equation}
\begin{cases}
	u_{tt}-\Delta u+a(x)b(t)u_t=f(u),\quad (t,x)\in (0,\infty)\times\mathbf{R}^n,\\
	u(0,x)=u_0(x),\quad u_t(0,x)=u_1(x),\quad x\in\mathbf{R}^n,
\end{cases}\label{eq11}
\end{equation}
where the coefficients of damping are
$$
	a(x)= a_0\langle x\rangle^{-\alpha},\quad b(t)=(1+t)^{-\beta},
	\quad \mbox{with}\ a_0>0,\alpha,\beta\ge 0, \alpha+\beta <1,
$$
where $\langle x\rangle=(1+|x|^2)^{1/2}$.
Here $u$ is a real-valued unknown function and
$(u_0,u_1)$ is in $H^1(\mathbf{R}^n)\times L^2(\mathbf{R}^n)$.
We note that $u_0$ and $u_1$ need not be compactly supported.
The nonlinear term $f(u)$ is given by
$$
	f(u)=\pm |u|^p\quad\mbox{or}\quad |u|^{p-1}u
$$
and the power $p$ satisfies
$$
	1<p\le\frac{n}{n-2}\quad (n\ge 3),\quad 1<p<\infty\quad (n=1,2).
$$
Our aim is to determine the critical exponent $p_c$, which is a number defined by the following property:

If $p_c<p$, all small data solutions of (\ref{eq11}) are global;
if $1<p\le p_c$, the time-local solution cannot be extended time-globally for some data.

It is expected that the critical exponent of (\ref{eq11}) is given by
$$
	p_c=1+\frac{2}{n-\alpha}.
$$
In this paper we shall prove the existence of global solutions with small data when
$p>1+2/(n-\alpha)$.
However, it is still open whether there exists a blow-up solution when $1<p\le 1+2/(n-\alpha)$.

When the damping term is missing and $f(u)=|u|^p$, that is
\begin{equation}
\begin{cases}
	u_{tt}-\Delta u=|u|^p,\quad (t,x)\in (0,\infty)\times\mathbf{R}^n,\\
	u(0,x)=u_0(x),\quad u_t(0,x)=u_1(x),\quad x\in\mathbf{R}^n,
\end{cases}\label{eq12}
\end{equation}
it is well known that the critical exponent $p_w(n)$ is the positive root of
$(n-1)p^2-(n+1)p-2=0$ for $n\ge 2\ (p_w(1)=\infty)$.
This is the famous Strauss conjecture and the proof is completed by the effort of many mathematicians
(see \cite{J,G,Si,St1,St2,Zhou,LS,GLS,R,T}).

For the linear wave equation with a damping term
\begin{equation}
\begin{cases}
	u_{tt}-\Delta u+c(t,x)u_t=0,\quad (t,x)\in (0,\infty)\times\mathbf{R}^n,\\
	u(0,x)=u_0(x),\quad u_t(0,x)=u_1(x),\quad x\in\mathbf{R}^n,
\end{cases}\label{eq13}
\end{equation}
there are many results about the asymptotic behavior of the solution.
When $c(t,x)=c_0>0$ and $(u_0, u_1)\in (H^1\cap L^1)\times(L^2\cap L^1)$,
Matsumura \cite{M} showed that the energy of solutions
decays at the same rate as the corresponding heat equation.
When the space dimension is $3$, using the exact expression of the solution,
Nishihara \cite{N1} discovered that the solution of (\ref{eq13}) with $c(t,x)=1$ is expressed asymptotically by
$$
	u(t,x)\sim v(t,x)+e^{-t/2}w(t,x),
$$
where $v(t,x)$ is the solution of the corresponding heat equation
$$
	\begin{cases}
	v_t-\Delta v=0,\quad (t,x)\in (0,\infty)\times\mathbf{R}^3,\\
	v(0,x)=u_0(x)+u_1(x),\quad x\in\mathbf{R}^3
	\end{cases}
$$
and $w(t,x)$ is the solution of the free wave equation
$$
	\begin{cases}
	w_{tt}-\Delta w=0,\quad (t,x)\in (0,\infty)\times\mathbf{R}^3,\\
	w(0,x)=u_0(x),\quad w_t(0,x)=u_1(x),\quad\in \mathbf{R}^3.
	\end{cases}
$$
These results indicate a diffusive structure of damped wave equations.
On the other hand, Mochizuki \cite{Mo} showed that if 
$0\le c(t,x)\le C(1+|x|)^{-1-\delta}$, where $\delta >0$, then
the energy of solutions of (\ref{eq13}) does not decay to $0$ for nonzero data and the solution is asymptotically free.
We can interpret this result as (\ref{eq13}) loses its \textquotedblleft parabolicity\textquotedblright
and recover its \textquotedblleft hyperbolicity\textquotedblright.
Wirth \cite{W1,W2} treated time-dependent damping case, that is $c(t,x)=b(t)$ in (\ref{eq13}).
By the Fourier transform method,
he got several sharp $L^p-L^q$ estimates of the solution and
showed that there exists diffusive structure
for general $b(t)$ including $b(t)=b_0(1+t)^{-\beta} (-1<\beta<1)$.
Todorova and Yordanov \cite{TY3} considered the case $c(t,x)=a(x)=a_0\langle x\rangle^{-\alpha}$
with $\alpha\in [0,1)$ and J. S. Kenigson and J. J. Kenigson \cite{KK} 
considered space-time dependent coefficient case
$c(t,x)=a(x)b(t), a(x)=a_0\langle x\rangle^{-\alpha}, b(t)=(1+t)^{-\beta}, (0\le\alpha+\beta<1)$.
They established the energy decay estimate
that also implies diffusive structure even in the decaying coefficient cases.
From these results, the decay rate $-1$ of the coefficient of the damping term is the threshold of parabolicity.
This is the reason why we assume $\alpha+\beta<1$ for (\ref{eq11}).
We mention that recently, Ikehata, Todorova and Yordanov \cite{ITY2}
treated the case $c(t,x)=a_0\langle x\rangle^{-1}$ and obtained almost optimal decay estimates.

There are also many results for the semilinear damped wave equation with absorbing semilinear term:
\begin{equation}
	\begin{cases}
	u_{tt}-\Delta u+a(x)b(t)u_t+|u|^{p-1}u=0,\quad (t,x)\in (0,\infty)\times\mathbf{R}^n,\\
	u(0,x)=u_0(x),\quad u_t(0,x)=u_1(x),\quad x\in\mathbf{R}^n,
	\end{cases}\label{eq14}
\end{equation}
It is well known that there exists a unique global solution even for large initial data.
When $a(x)b(t)=1$, that is constant coefficient case, 
Kawashima, Nakao and Ono \cite{KNO},
Karch \cite{K}, 
Hayashi, Kaikina and Naumkin \cite{HKN4},
Ikehata, Nishihara and Zhao \cite{INZ}
and Nishihara \cite{N2}
showed global existence of solutions and
that their asymptotic profile is given by a constant multiple of the Gauss kernel for
$1+2/n<p$ and $n\le 4$.
For $1<p\le 1+2/n$,
Nishihara and Zhao \cite{NZhao},
Ikehata, Nishihara and Zhao \cite{INZ},
Nishihara \cite{N2}
proved that the decay rate of the solution agrees with
that of a self-similar solution of the corresponding heat equation. 
Hayashi, Kaikina and Naumkin \cite{HKN1,HKN2,HKN3,HKN4}
proved the large time asymptotic formulas in terms of the weighted Sobolev spaces.
These results indicate the critical exponent for (\ref{eq14}) with $a(x)b(t)=1$ is given by $p_c=1+\frac{2}{n}$.
In this case the critical exponent means the turning point of the asymptotic behavior of the solution.
When $b(t)=1, a(x)=\langle x\rangle^{-\alpha} (0\le\alpha<1)$,
namely space-dependent damping case, Nishihara \cite{N3} established
decay estimates of solutions and conjectured the critical exponent is given by
$p_c=1+2/(n-\alpha)$.
When $a(x)=1, b(t)=(1+t)^{-\beta} (-1<\beta<1)$, Nishihara and Zhai \cite{NZhai} proved decay estimates
of solutions and conjectured the critical exponent is $p_c=1+2/n$.
Finally in the case $a(x)=\langle x\rangle^{-\alpha}, b(t)=(1+t)^{-\beta} (0\le \alpha+\beta<1)$, 
Lin, Nishihara and Zhai \cite{LNZ1,LNZ2} showed decay estimates of the solution
and conjectured the critical exponent is $p_c=1+2/(n-\alpha)$.
They used a weighted energy method, which is originally developed by Todorova and Yordanov \cite{TY1,TY2}.
In this paper we shall essentially use the techniques and method that they used.

Li and Zhou \cite{LiZ} considered the semilinear damped wave equation
\begin{equation}
	u_{tt}-\Delta u+u_t=|u|^p.\label{eq15}
\end{equation}
They proved that if $n\le 2, 1<p\le 1+\frac{2}{n}$ and the data are positive on average,
then the local solution of (\ref{eq15}) must blow up in a finite time.
Todorova and Yordanov \cite{TY1,TY2}
developed a weighted energy method using the function which has the form $e^{2\psi}$
and determined that the critical exponent of (\ref{eq15}) is
$$
	p_c=1+\frac{2}{n},
$$
which is well known as Fujita's critical exponent for the heat equation
$u_t-\Delta u=u^{p}$ (see \cite{F}).
More precisely, they proved small data global existence in the case $p>1+2/n$
and blow-up for all solutions of (\ref{eq15}) with positive on average data in the case $1<p<1+2/n$.
Later on Zhang \cite{Zhang} showed that the critical exponent $p=1+2/n$ belongs to the blow-up region.
We mention that Todorova and Yordanov \cite{TY1,TY2}
assumed data have compact support and essentially used this property.
However, Ikehata and Tanizawa \cite{IT} removed this assumption.
Ikehata, Todorova and Yordanov \cite{ITY} investigated the space-dependent coefficient case:
\begin{equation}
	u_{tt}-\Delta u+a(x)u_t=|u|^p,\label{eq16}
\end{equation}
where
$$
	a(x)\sim a_0\langle x\rangle^{-\alpha}, |x|\rightarrow \infty,
	\quad \mbox{radially symmetric and} \ 0\le\alpha <1.
$$
They proved that the critical exponent of (\ref{eq15}) is given by
$$
	p_c=1+\frac{2}{n-\alpha}
$$
by using a refined multiplier method. Their method also depends on the finite propagation speed property. 
Recently, Nishihara \cite{N4} and Lin, Nishihara and Zhai \cite{LNZ2}
considered the semilinear wave equation with time-dependent damping
\begin{equation}
	u_{tt}-\Delta u+b(t)u_t=|u|^p,\label{eq17}
\end{equation}
where
$$
	b(t)=b_0(1+t)^{-\beta}, \quad \beta\in (-1,1).
$$
They proved that the critical exponent of (\ref{eq17}) is
$$
	p_c=1+\frac{2}{n}.
$$
This shows that, roughly speaking, time-dependent coefficients of damping term do not 
influence the critical exponent.
Therefore we expect that the critical exponent of the semilinear wave equation (\ref{eq11}) is
$$
	p_c=1+\frac{2}{n-\alpha}.
$$
To state our results, we introduce an auxiliary function
\begin{equation}
	\psi(t,x):=A\frac{\langle x\rangle^{2-\alpha}}{(1+t)^{1+\beta}}\label{eq22}
\end{equation}
with
\begin{equation}
	A=\frac{(1+\beta)a_0}{(2-\alpha)^2(2+\delta)},\quad \delta>0\label{eq23}
\end{equation}
This type of weight function is first introduced by Ikehata and Tanizawa \cite{IT}.
We have the following result:
%%%%%%%%%%%%  Global Existence  %%%%%%%%%%%%%%
\begin{theorem}\label{thm1}
If
$$
	p>1+\frac{2}{n-\alpha},
$$
then there exists a small positive number $\delta_0>0$ such that
for any $0<\delta\le\delta_0$ the following holds:
If
$$
	I_0^2:=\int_{\mathbf{R}^n}
		e^{2\psi(0,x)}
		(u_1^2+|\nabla u_0|^2+|u_0|^2)dx
$$
is sufficiently small,
then there exists a unique solution
$u\in C([0,\infty);H^1(\mathbf{R}^n))\cap C^1([0,\infty);L^2(\mathbf{R}^n))$
to $(\ref{eq11})$ satisfying
\begin{eqnarray}
	\int_{\mathbf{R}^n}e^{2\psi(t,x)}|u(t,x)|^2dx
	&\le& C_{\delta}(1+t)^{-(1+\beta)\frac{n-2\alpha}{2-\alpha}+\varepsilon},\label{eq18}\\
	\int_{\mathbf{R}^n}e^{2\psi(t,x)}(|u_t(t,x)|^2+|\nabla u(t,x)|^2)dx
	&\le&C_{\delta}
		(1+t)^{-(1+\beta)(\frac{n-\alpha}{2-\alpha}+1)+\varepsilon},\nonumber
\end{eqnarray}
where 
\begin{equation}
	\varepsilon=\varepsilon(\delta):=\frac{3(1+\beta)(n-\alpha)}{2(2-\alpha)(2+\delta)}\delta\label{eqepsilon}
\end{equation}
and $C_{\delta}$ is a constant depending on $\delta$.
\end{theorem}
\begin{remark}
When $1<p\le 1+2/(n-\alpha)$, it is expected that no matter how small the data are,
if the data have some shape,
then the corresponding local solution blows up in finite time.
However, we have no result.
\end{remark}
\begin{remark}
We do not assume that the data are compactly supported.
Hence our result is an extension of the results of Ikehata, Todorova and Yordanov \cite{ITY}
to noncompactly supported data  cases.
However, we prove only the case $a(x)=a_0\langle x\rangle^{-\alpha}$.
\end{remark}

As a consequence of the main theorem, we have an exponential decay estimate outside a parabolic region.
\begin{corollary}\label{cor13}
If
$$
	p>1+\frac{2}{n-\alpha},
$$
then there exists a small positive number $\delta_0>0$ such that for any $0<\delta\le\delta_0$ the following holds:
Take $\rho$ and $\mu$ so small that
$$
	0<\rho<1-\alpha-\beta,\quad\mbox{and}
	\quad 0<\mu<2A,
$$
and put
$$
	\Omega_{\rho}(t):=\{x\in\mathbf{R}^n; \langle x\rangle^{2-\alpha}\ge (1+t)^{1+\beta+\rho}\}.
$$
Then, for the global solution $u$ in Theorem \ref{thm1}, we have the following estimate
\begin{equation}
	\int_{\Omega_{\rho}(t)}(u_t^2+|\nabla u|^2+u^2)dx
	\le C_{\delta,\rho,\mu}(1+t)^{-\frac{(1+\beta)(n-2\alpha)}{2-\alpha}+\varepsilon}
		e^{-(2A-\mu)(1+t)^{\rho}},
\end{equation}
here $\varepsilon$ is defined by $(\ref{eqepsilon})$ and
$C_{\delta,\rho,\mu}$ is a constant depending on $\delta, \rho$ and $\mu$.
\end{corollary}
Namely, the decay rate of solution in the region $\Omega_{\rho}(t)$ is exponential.
We note that the support of $u(t)$ and the region $\Omega_{\rho}(t)$
can intersect even if the data are compactly supported.
This phenomenon was first discovered by Todorova and Yordanov \cite{TY2}.
We can interpret this result as follows:
The support of the solution is strongly suppressed by damping,
so that the solution is concentrated in the parabolic region much smaller than the light cone.

%%%%%%%%%%%%%%%%%%%%%%%%%%%%%%%%%%%%
%%%%%%%%%%%%%%%%  Section2  %%%%%%%%%%%%%%
%%%%%%%%%%%%%%%%%%%%%%%%%%%%%%%%%%%%

\section{Proof of Theorem \ref{thm1}}
In this section we prove our main result.
At first we prepare some notation and terminology.
We put
$$
	\|f\|_{L^p(\mathbf{R}^n)}:=
	\left(\int_{\mathbf{R}^n}|f(x)|^pdx\right)^{1/p},\quad \|u\|:=\|u\|_{L^2(\mathbf{R}^n)}.
$$
By $H^1(\mathbf{R}^n)$ we denote the usual Sobolev space.
For an interval $I$ and a Banach space $X$, we define $C^r(I;X)$ as the Banach space whose
element is an $r$-times continuously differentiable mapping from $I$ to $X$ with respect to the topology in $X$.
The letter $C$ indicates the generic constant, which may change from line to the next line. 

To prove Theorem \ref{thm1}, we use a weighted energy method
which was originally developed by Todorova and Yordanov \cite{TY1, TY2}.
We first describe the local existence:
%%%%%%%%%%%  Local Existence %%%%%%%%%%%%%%%%
\begin{proposition}\label{prop21}
For any $\delta>0$, there exists $T_m\in (0,+\infty]$ depending on $I_0^2$ such that
the Cauchy problem $(\ref{eq11})$
has a unique solution $u\in C([0,T_m);H^1(\mathbf{R}^n))\cap C^1([0,T_m);L^2(\mathbf{R}^n))$,
and if $T_m <+\infty$ then we have
$$
	\liminf_{t\to T_m}\int_{\mathbf{R}^n}
	e^{\psi(t,x)}
		(u_t^2+|\nabla u|^2+u^2)dx=+\infty.$$
\end{proposition}
We can prove this proposition by standard arguments (see \cite{IT}).
We prove a priori estimate for the following functional:
\begin{eqnarray}
	M(t)&:=&
	\sup_{0\le\tau <t}\left\{(1+\tau)^{B+1-\varepsilon}\int_{\mathbf{R}^n}e^{2\psi}(u_t^2+|\nabla u|^2)dx
		\right.\nonumber\\
	&&\left.+(1+\tau)^{B-\varepsilon}\int_{\mathbf{R}^n}e^{2\psi}a(x)b(t)u^2dx\right\},\label{eq21}
\end{eqnarray}
where
$$
	B:=\frac{(1+\beta)(n-\alpha)}{2-\alpha}+\beta
$$
and $\varepsilon$ is given by (\ref{eqepsilon}).
From (\ref{eq22}), (\ref{eq23}), it is easy to see that
\begin{eqnarray}
	-\psi_t&=&\frac{1+\beta}{1+t}\psi,\label{eq24}\\
	\nabla\psi&=&A\frac{(2-\alpha)\langle x\rangle^{-\alpha}x}{(1+t)^{1+\beta}},\label{eq25}\\
	\Delta\psi&=&A(2-\alpha)(n-\alpha)\frac{\langle x\rangle^{-\alpha}}{(1+t)^{1+\beta}}
	+A(2-\alpha)\alpha\frac{\langle x\rangle^{-2-\alpha}}{(1+t)^{1+\beta}}\nonumber\\
	&\ge&\frac{(1+\beta)(n-\alpha)}{(2-\alpha)(2+\delta)}\frac{a(x)b(t)}{1+t}\nonumber\\
	&=:&\left( \frac{(1+\beta)(n-\alpha)}{2(2-\alpha)}-\delta_1\right)\frac{a(x)b(t)}{1+t}.\label{eq26}
\end{eqnarray}
Here and after,
$\delta_i(i=1,2,\ldots)$
is a positive constant depending only on
$\delta$
such that
$$
	\delta_i\rightarrow 0^+ \quad \mbox{as} \quad\delta\rightarrow 0^+.
$$
We also have
\begin{eqnarray}
	(-\psi_t)a(x)b(t)&=&
		Aa_0(1+\beta)\frac{\langle x\rangle^{2-2\alpha}}{(1+t)^{2+2\beta}}\nonumber\\
	&\ge&\frac{a_0(1+\beta)}{(2-\alpha)^2A}
		A^2(2-\alpha)^2\frac{\langle x\rangle^{-2\alpha}|x|^2}{(1+t)^{2+2\beta}}\nonumber\\
	&=&(2+\delta)|\nabla\psi|^2.\label{eq27}
\end{eqnarray}
By multiplying (\ref{eq11}) by $e^{2\psi}u_t$, it follows that
\begin{eqnarray}
	\lefteqn{\frac{\partial}{\partial t}\left[ \frac{e^{2\psi}}{2}(u_t^2+|\nabla u|^2)\right]
		-\nabla\cdot(e^{2\psi}u_t\nabla u)}\nonumber\\
	&&+e^{2\psi}\left(a(x)b(t)-\frac{|\nabla\psi|^2}{-\psi_t}-\psi_t\right)u_t^2
		+\underbrace{\frac{e^{2\psi}}{-\psi_t}|\psi_t\nabla u-u_t\nabla\psi|^2}_{T_1}\nonumber\\
	&&=\frac{\partial}{\partial t}\left[e^{2\psi}F(u)\right]+2e^{2\psi}(-\psi_t)F(u),\label{eq28}
\end{eqnarray}
where $F$ is the primitive of $f$ satisfying $F(0)=0$, namely
$F^{\prime}(u)=f(u)$.
Using the Schwarz inequality and (\ref{eq27}), we can calculate
\begin{eqnarray*}
	T_1&=&
	\frac{e^{2\psi}}{-\psi_t}(\psi_t^2|\nabla u|^2-2\psi_tu_t\nabla u\cdot\nabla\psi+u_t^2|\nabla\psi|^2)\\
	&\ge&\frac{e^{2\psi}}{-\psi_t}
		\left(\frac{1}{5}\psi_t^2|\nabla u|^2-\frac{1}{4}u_t^2|\nabla\psi|^2\right)\\
	&\ge&e^{2\psi}
		\left(\frac{1}{5}(-\psi_t)|\nabla u|^2-\frac{a(x)b(t)}{4(2+\delta)}u_t^2\right).
\end{eqnarray*}
From this and (\ref{eq27}), we obtain
\begin{eqnarray}
	\lefteqn{\frac{\partial}{\partial t}\left[ \frac{e^{2\psi}}{2}(u_t^2+|\nabla u|^2)\right]
		-\nabla\cdot(e^{2\psi}u_t\nabla u)}\nonumber\\
	&&+e^{2\psi}
		\left\{\left(\frac{1}{4}a(x)b(t)-\psi_t\right)u_t^2+\frac{-\psi_t}{5}|\nabla u|^2\right\}\nonumber\\
	&&\le\frac{\partial}{\partial t}
		\left[e^{2\psi}F(u)\right]+2e^{2\psi}(-\psi_t)F(u).\label{eq29}
\end{eqnarray}
By multiplying (\ref{eq29}) by $(t_0+t)^{B+1-\varepsilon}$, here $t_0\ge 1$ is determined later, it follows that
\begin{eqnarray}
	\lefteqn{\frac{\partial}{\partial t}
		\left[(t_0+t)^{B+1-\varepsilon}\frac{e^{2\psi}}{2}(u_t^2+|\nabla u|^2)\right]}\nonumber\\
	&&-(B+1-\varepsilon)(t_0+t)^{B-\varepsilon}\frac{e^{2\psi}}{2}(u_t^2+|\nabla u|^2)\nonumber\\
	&&-\nabla\cdot ((t_0+t)^{B+1-\varepsilon}e^{2\psi}u_t\nabla u)\nonumber\\
	&&+e^{2\psi}(t_0+t)^{B+1-\varepsilon}
		\left\{\left(\frac{1}{4}a(x)b(t)-\psi_t\right)u_t^2+\frac{-\psi_t}{5}|\nabla u|^2\right\}\nonumber\\
	&\le&\frac{\partial}{\partial t}\left[(t_0+t)^{B+1-\varepsilon}e^{2\psi}F(u)\right]
		-(B+1-\varepsilon)(t_0+t)^{B-\varepsilon}e^{2\psi}F(u)\nonumber\\
	&&+2(t_0+t)^{B+1-\varepsilon}e^{2\psi}(-\psi_t)F(u).\label{eq210}
\end{eqnarray}
We put
$$
	\begin{array}{ll}
		\displaystyle E(t):=\int_{\mathbf{R}^n}e^{2\psi}(u_t^2+|\nabla u|^2)dx,
		\displaystyle&E_{\psi}(t):=\displaystyle\int_{\mathbf{R}^n}e^{2\psi}(-\psi_t)(u_t^2+|\nabla u|^2)dx,\\
		\displaystyle J(t;g):=\int_{\mathbf{R}^n}e^{2\psi}gdx,
		\displaystyle&J_{\psi}(t;g):=\displaystyle\int_{\mathbf{R}^n}e^{2\psi}(-\psi_t)gdx.
	\end{array}
$$
Integrating (\ref{eq210}) over the whole space, we have
\begin{eqnarray}
	\lefteqn{\frac{1}{2}\frac{d}{dt}\left[(t_0+t)^{B+1-\varepsilon}E(t)\right]
		-\frac{1}{2}(B+1-\varepsilon)(t_0+t)^{B-\varepsilon}E(t)}\nonumber\\
	&&+\frac{1}{4}(t_0+t)^{B+1-\varepsilon}J(t,a(x)b(t)u_t^2)
		+\frac{1}{5}(t_0+t)^{B+1-\varepsilon}E_{\psi}(t)\nonumber\\
	&\le&\frac{d}{dt}\left[(t_0+t)^{B+1-\varepsilon}\int e^{2\psi}F(u)dx\right]\nonumber\\
	&&+C(t_0+t)^{B+1-\varepsilon}J_{\psi}(t; |u|^{p+1})
		+C(t_0+t)^{B-\varepsilon}J(t; |u|^{p+1})\label{eq211}
\end{eqnarray}
Therefore, we integrate (\ref{eq211}) on the interval $[0,t]$ and obtain the estimate for $(t_0+t)^{B+1-\varepsilon}E(t)$,
which is the first term of $M(t)$:
\begin{eqnarray}
	\lefteqn{(t_0+t)^{B+1-\varepsilon}E(t)-C\int_0^t(t_0+\tau)^{B-\varepsilon}E(\tau)d\tau}\nonumber\\
	&&+\int_0^t(t_0+\tau)^{B+1-\varepsilon}J(\tau;a(x)b(t)u_t^2)
		+(t_0+\tau)^{B+1-\varepsilon}E_{\psi}(\tau)d\tau\nonumber\\
	&\le&CI_0^2+C(t_0+t)^{B+1-\varepsilon}J(t;|u|^{p+1})\nonumber\\
	&&+C\int_0^t(t_0+\tau)^{B+1-\varepsilon}J_{\psi}(\tau;|u|^{p+1})d\tau\nonumber\\
	&&+C\int_0^t(t_0+t)^{B-\varepsilon}J(\tau;|u|^{p+1})d\tau.\label{eq212}
\end{eqnarray}
In order to complete a priori estimate, however, we have to manage the second term of the inequality above
whose sign is negative,
and we also have to estimate the second term of $M(t)$.
The following argument, which is little more complicated, can settle both these problems.

At first, we multiply (\ref{eq11}) by $e^{2\psi}u$ and have
\begin{eqnarray}
\lefteqn{\frac{\partial}{\partial t}\left[e^{2\psi}\left(uu_t+\frac{a(x)b(t)}{2}u^2\right)\right]
		-\nabla\cdot (e^{2\psi}u\nabla u)}\nonumber\\
&&+e^{2\psi}\left\{|\nabla u|^2+\left(-\psi_t+\frac{\beta}{2(1+t)}\right)a(x)b(t)u^2
	+\underbrace{2u\nabla\psi\cdot\nabla u}_{T_2}-2\psi_tuu_t-u_t^2\right\}\nonumber\\
&&=e^{2\psi}uf(u).\label{eq213}
\end{eqnarray}
We calculate
\begin{eqnarray*}
	e^{2\psi}T_2&=&4e^{2\psi}u\nabla\psi\cdot\nabla u-2e^{2\psi}u\nabla\psi\cdot\nabla u\\
	&=&4e^{2\psi}u\nabla\psi\cdot\nabla u
		-\nabla\cdot(e^{2\psi}u^2\nabla\psi)+2e^{2\psi}u^2|\nabla\psi|^2+e^{2\psi}(\Delta\psi)u^2
\end{eqnarray*}
and by (\ref{eq26}) we can rewrite (\ref{eq213}) to
\begin{eqnarray}
	\lefteqn{\frac{\partial}{\partial t}\left[e^{2\psi}\left(uu_t+\frac{a(x)b(t)}{2}u^2\right)\right]
		-\nabla\cdot(e^{2\psi}(u\nabla u+u^2\nabla\psi))}\nonumber\\
	&&+e^{2\psi}\Big\{\underbrace{|\nabla u|^2+4u\nabla u\cdot\nabla\psi
		+(-\psi_t)a(x)b(t)+2|\nabla\psi|^2)u^2}_{T_3}\nonumber\\
	&&+(B-2\delta_1)\frac{a(x)b(t)}{2(1+t)}u^2-2\psi_tuu_t-u_t^2\Big\} \le e^{2\psi}uf(u).\label{eq214}
\end{eqnarray}
It follows from (\ref{eq27}) that
\begin{eqnarray*}
T_3&=&
|\nabla u|^2+4u\nabla u\cdot\nabla\psi\\
&&+\left\{\left(1-\frac{\delta}{3}\right)(-\psi_t)a(x)b(t)+2|\nabla\psi|^2\right\}u^2+\frac{\delta}{3}(-\psi_t)a(x)b(t)u^2\\
&\ge&|\nabla u|^2+4u\nabla u\cdot\nabla\psi\\
&&+\left(4+\frac{\delta}{3}-\frac{\delta^2}{3}\right)|\nabla\psi|^2u^2+\frac{\delta}{3}(-\psi_t)a(x)b(t)u^2\\
&=&\left(1-\frac{4}{4+\delta_2}\right)|\nabla u|^2+\delta_2|\nabla\psi|^2u^2\\
&&+\left| \frac{2}{\sqrt{4+\delta_2}}\nabla u+\sqrt{4+\delta_2}u\nabla\psi\right|^2
	+\frac{\delta}{3}(-\psi_t)a(x)b(t)u^2\\
&\ge&\delta_3(|\nabla u|^2+|\nabla\psi|^2u^2)+\frac{\delta}{3}(-\psi_t)a(x)b(t)u^2,
\end{eqnarray*}
where
$$\delta_2:=\frac{\delta}{6}-\frac{\delta^2}{6},
	\quad\delta_3:=\min(1-\frac{4}{4+\delta_2}, \delta_2).$$
Thus, we obtain
\begin{eqnarray}
	\lefteqn{\frac{\partial}{\partial t}\left[e^{2\psi}\left(uu_t+\frac{a(x)b(t)}{2}u^2\right)\right]
		-\nabla\cdot(e^{2\psi}(u\nabla u+u^2\nabla\psi))}\nonumber\\
	&&+e^{2\psi}\delta_3|\nabla u|^2\nonumber\\
	&&+e^{2\psi}\left(\delta_3|\nabla\psi|^2+\frac{\delta}{3}(-\psi_t)a(x)b(t)
		+(B-2\delta_1)\frac{a(x)b(t)}{2(1+t)}\right)u^2\nonumber\\
	&&+e^{2\psi}(-2\psi_tuu_t-u_t^2)\nonumber\\
	&\le&e^{2\psi}uf(u).\label{eq215}
\end{eqnarray}
Following Nishihara \cite{LNZ1}, related to the size of $1+|x|^2$ and the size of $(1+t)^2$,
we divide the space $\mathbf{R}^n$ into two different zones $\Omega(t;K,t_0)$ and $\Omega^c(t;K,t_0)$,
where
$$
	\Omega=\Omega(t;K,t_0):=\{x\in\mathbf{R}^n; (t_0+t)^2\ge K+|x|^2\},
$$
and $\Omega^c=\Omega^c(t;K,t_0):=\mathbf{R}^n\setminus\Omega(t;K,t_0)$
with $K\ge 1$ determined later.
Since $a(x)b(t)\ge a_0(t+t_0)^{-(\alpha+\beta)}$ in the domain $\Omega$,
we multiply (\ref{eq29}) by $(t_0+t)^{\alpha+\beta}$ and obtain
\begin{eqnarray}
	\lefteqn{\frac{\partial}{\partial t}\left[\frac{e^{2\psi}}{2}(t_0+t)^{\alpha+\beta}(u_t^2+|\nabla u|^2)\right]
		-\nabla\cdot (e^{2\psi}(t_0+t)^{\alpha+\beta}u_t\nabla u)}\nonumber\\
	&&+e^{2\psi}\left[ \left(\frac{a_0}{4}-\frac{\alpha+\beta}{2(t_0+t)^{1-\alpha-\beta}}\right)
		+(t_0+t)^{\alpha+\beta}(-\psi_t)\right] u_t^2\nonumber\\
	&&+e^{2\psi}\left[ \frac{-\psi_t}{5}(t_0+t)^{\alpha+\beta}
		-\frac{\alpha+\beta}{2(t_0+t)^{1-\alpha-\beta}}\right] |\nabla u|^2\nonumber\\
	&\le&\frac{\partial}{\partial t}[(t_0+t)^{\alpha+\beta}e^{2\psi}F(u)]
		-\frac{\alpha+\beta}{(t_0+t)^{1-\alpha-\beta}}e^{2\psi}F(u)\nonumber\\
	&&+2(t_0+t)^{\alpha+\beta}e^{2\psi}(-\psi_t)F(u).\label{eq216}
\end{eqnarray}
Let $\nu$ be a small positive number depends on $\delta$, which will be chosen later.
By (\ref{eq216})$+\nu$(\ref{eq215}), we have
\begin{eqnarray}
	\lefteqn{\frac{\partial}{\partial t}\left[ e^{2\psi}\left(\frac{(t_0+t)^{\alpha+\beta}}{2}u_t^2
		+\nu uu_t+\frac{\nu a(x)b(t)}{2}u^2+
		\frac{(t_0+t)^{\alpha+\beta}}{2}|\nabla u|^2\right)\right]}\nonumber\\
	&&-\nabla\cdot (e^{2\psi}(t_0+t)^{\alpha+\beta}u_t\nabla u+\nu e^{2\psi}(u\nabla u+u^2\nabla\psi))\nonumber\\
	&&+e^{2\psi}\left[\left(\frac{a_0}{4}-\frac{\alpha+\beta}{2(t_0+t)^{1-\alpha-\beta}}-\nu\right)
		+(t_0+t)^{\alpha+\beta}(-\psi_t)\right] u_t^2\nonumber\\
	&&+e^{2\psi}\left[ \nu\delta_3-\frac{\alpha+\beta}{2(t_0+t)^{1-\alpha-\beta}}
		+\frac{-\psi_t}{5}(t_0+t)^{\alpha+\beta}\right] |\nabla u|^2\nonumber\\
	&&+e^{2\psi}\left[ \nu (\delta_3|\nabla\psi|^2+\frac{\delta}{3}(-\psi_t)a(x)b(t)
		+(B-2\delta_1)\frac{a(x)b(t)}{2(1+t)}\right] u^2\nonumber\\
	&&+2\nu e^{2\psi}(-\psi_t)uu_t\nonumber\\
	&\le&\frac{\partial}{\partial t}[(t_0+t)^{\alpha+\beta}e^{2\psi}F(u)]
		-\frac{\alpha+\beta}{(t_0+t)^{1-\alpha-\beta}}e^{2\psi}F(u)\nonumber\\
	&&+2(t_0+t)^{\alpha+\beta}e^{2\psi}(-\psi_t)F(u)+\nu e^{2\psi}uf(u).\label{eq217}
\end{eqnarray}
By the Schwarz inequality, the last term of the left hand side in the above inequality can be estimated as
$$
	|2\nu (-\psi_t)uu_t|
	\le\frac{\nu\delta}{3}(-\psi_t)a(x)b(t)u^2+\frac{3\nu}{a_0\delta}(-\psi_t)(t_0+t)^{\alpha+\beta}u_t^2.
$$
Thus, we have
\begin{eqnarray}
	\lefteqn{\frac{\partial}{\partial t}\left[ e^{2\psi}\left(\frac{(t_0+t)^{\alpha+\beta}}{2}u_t^2
		+\nu uu_t+\frac{\nu a(x)b(t)}{2}u^2+\frac{(t_0+t)^{\alpha+\beta}}{2}|\nabla u|^2\right)\right]}\nonumber\\
	&&-\nabla\cdot (e^{2\psi}(t_0+t)^{\alpha+\beta}u_t\nabla u+\nu e^{2\psi}(u\nabla u+u^2\nabla\psi))\nonumber\\
	&&+e^{2\psi}\left[\left(\frac{a_0}{4}-\frac{\alpha+\beta}{2(t_0+t)^{1-\alpha-\beta}}-\nu\right)
		+\left(1-\frac{3\nu}{a_0\delta}\right)(t_0+t)^{\alpha+\beta}(-\psi_t)\right] u_t^2\nonumber\\
	&&+e^{2\psi}\left[\nu\delta_3-\frac{\alpha+\beta}{2(t_0+t)^{1-\alpha-\beta}}
		+\frac{-\psi_t}{5}(t_0+t)^{\alpha+\beta}\right] |\nabla u|^2\nonumber\\
	&&+e^{2\psi}\left[ \nu \left(\delta_3|\nabla\psi|^2
		+(B-2\delta_1)\frac{a(x)b(t)}{2(1+t)}\right)\right] u^2\nonumber\\
	&\le&\frac{\partial}{\partial t}[(t_0+t)^{\alpha+\beta}e^{2\psi}F(u)]
		-\frac{\alpha+\beta}{(t_0+t)^{1-\alpha-\beta}}e^{2\psi}F(u)\nonumber\\
	&&+2(t_0+t)^{\alpha+\beta}e^{2\psi}(-\psi_t)F(u)+\nu e^{2\psi}uf(u).\label{eq218}
\end{eqnarray}
Now we choose the parameters $\nu$ and $t_0$ such that
\begin{eqnarray*}
\frac{a_0}{4}-\frac{\alpha+\beta}{2(t_0+t)^{1-\alpha-\beta}}-\nu\ge c_0,\quad
1-\frac{3\nu}{a_0\delta}\ge c_0,\\
\nu\delta_3-\frac{\alpha+\beta}{2(t_0+t)^{1-\alpha-\beta}}\ge c_0,\quad
\nu\delta_3\ge c_0,\quad
\frac{1}{5}\ge c_0,
\end{eqnarray*}
hold for some constant $c_0>0$.
This is possible because we first determine $\nu$ sufficiently small depending on $\delta$ and then
we choose $t_0$ sufficiently large depending on $\nu$.
Therefore, integrating (\ref{eq218}) on $\Omega$, we obtain the following energy inequality:
\begin{equation}
\frac{d}{dt}\overline{E}_{\psi}(t;\Omega(t;K,t_0))-N_1(t)-M_1(t)+H_{\psi}(t;\Omega(t;K,t_0))\le P_1,
\label{eq219}
\end{equation}
where
\begin{eqnarray*}
	\overline{E}_{\psi}(t;\Omega)&=&\overline{E}_{\psi}(t;\Omega(t;K,t_0))\\
		&:=&\int_{\Omega}e^{2\psi}\left(\frac{(t_0+t)^{\alpha+\beta}}{2}u_t^2
		+\nu uu_t+\frac{\nu a(x)b(t)}{2}u^2+\frac{(t_0+t)^{\alpha+\beta}}{2}|\nabla u|^2\right)dx,
\end{eqnarray*}
\begin{eqnarray*}
	N_1(t)&:=&\int_{\mathbf{S}^{n-1}}e^{2\psi}
		\left(\frac{(t_0+t)^{\alpha+\beta}}{2}u_t^2
		+\nu uu_t+\frac{\nu a(x)b(t)}{2}u^2\right.\\
		&&+\left.\left.\frac{(t_0+t)^{\alpha+\beta}}{2}|\nabla u|^2\right)\right|_{|x|=\sqrt{(t_0+t)^2-K}}\\
		&&\times [(t_0+t)^2-K]^{(n-1)/2}d\theta\cdot\frac{d}{dt}\sqrt{(t_0+t)^2-K},
\end{eqnarray*}
$$
	M_1(t):=\int_{\partial\Omega}(e^{2\psi}(t_0+t)^{\alpha+\beta}u_t\nabla u
		+\nu e^{2\psi}(u\nabla u+u^2\nabla\psi))\cdot\vec{n}dS,
$$
\begin{eqnarray*}
	H_{\psi}(t;\Omega)&=&H_{\psi}(t;\Omega(t;K,t_0))\\
		&:=&c_0\int_{\Omega}e^{2\psi}(1+(t_0+t)^{\alpha+\beta}(-\psi_t))(u_t^2+|\nabla u|^2)dx\\
		&&+\nu (B-2\delta_1)\int_{\Omega}\frac{e^{2\psi}a(x)b(t)}{2(1+t)}u^2dx,
\end{eqnarray*}
\begin{eqnarray*}
	P_1&:=&\frac{d}{dt}\left[(t_0+t)^{\alpha+\beta}\int_{\Omega}e^{2\psi}F(u)dx\right]\\
		&&-\int_{\mathbf{S}^{n-1}}(t_0+t)^{\alpha+\beta}e^{2\psi}F(u)\Big|_{|x|=\sqrt{(t_0+t)^2-K}}\\
		&&\qquad\times [(t_0+t)^2-K]^{(n-1)/2}d\theta\cdot\frac{d}{dt}\sqrt{(t_0+t)^2-K}\\
		&&+C\int_{\Omega}e^{2\psi}(1+(t_0+t)^{\alpha+\beta}(-\psi_t))|u|^{p+1}dx.
\end{eqnarray*}
Here $\vec{n}$ denotes the unit outer normal vector of $\partial \Omega$.
We note that by $\nu\le a_0/4$ and
$$
	|\nu uu_t|\le\frac{\nu a(x)b(t)}{4}u^2+\frac{\nu(t_0+t)^{\alpha+\beta}}{a_0}u_t^2,
$$
it follows that
\begin{eqnarray*}
\lefteqn{c\int_{\Omega}e^{2\psi}(t_0+t)^{\alpha+\beta}(u_t^2+|\nabla u|^2)dx
	+c\int_{\Omega}e^{2\psi}a(x)b(t)u^2dx}\\
&\le&\overline{E}_{\psi}(t;\Omega(t;K,t_0))\\
&\le&C\int_{\Omega}e^{2\psi}(t_0+t)^{\alpha+\beta}(u_t^2+|\nabla u|^2)dx
	+C\int_{\Omega}e^{2\psi}a(x)b(t)u^2dx
\end{eqnarray*}
for some constants $c>0$ and $C>0$.

Next, we derive an energy inequality in the domain $\Omega^c$.
We use the notation
$$\langle x\rangle_K:=(K+|x|^2)^{1/2}.$$
Since $a(x)b(t)\ge a_0\langle x\rangle_K^{-(\alpha+\beta)}$ in $\Omega^c(t,;K,t_0)$,
we multiply (\ref{eq29}) by $\langle x\rangle_K^{\alpha+\beta}$ and obtain
\begin{eqnarray}
	\lefteqn{\frac{\partial}{\partial t}\left[\frac{e^{2\psi}}{2}\langle x\rangle_K^{\alpha+\beta}
		(u_t^2+|\nabla u|^2)\right] -\nabla\cdot (e^{2\psi}\langle x\rangle_K^{\alpha+\beta}u_t\nabla u)}\nonumber\\
	&&+e^{2\psi}\left(\frac{a_0}{4}+(-\psi_t)\langle x\rangle_K^{\alpha+\beta}\right)u_t^2
		+\frac{1}{5}e^{2\psi}(-\psi_t)\langle x\rangle_K^{\alpha+\beta}|\nabla u|^2\nonumber\\
	&&+(\alpha+\beta)e^{2\psi}\langle x\rangle_K^{\alpha+\beta-2}x\cdot u_t\nabla u\nonumber\\
	&\le&\frac{\partial}{\partial t}[e^{2\psi}\langle x\rangle_K^{\alpha+\beta}F(u)]
		+2e^{2\psi}\langle x\rangle_K^{\alpha+\beta}(-\psi_t)F(u).\label{eq220}
\end{eqnarray}
By $(\ref{eq220})+\hat{\nu}\times (\ref{eq215})$, here $\hat{\nu}$ is a small positive parameter determined later,
it follows that
\begin{eqnarray}
	\lefteqn{\frac{\partial}{\partial t}\left[e^{2\psi}\left(\frac{\langle x\rangle_K^{\alpha+\beta}}{2}u_t^2
		+\hat{\nu}uu_t+\frac{\hat{\nu}a(x)b(t)}{2}u^2+\frac{\langle x\rangle_K^{\alpha+\beta}}{2}|\nabla u|^2
		\right)\right] }\nonumber\\
	&&-\nabla\cdot(e^{2\psi}\langle x\rangle_K^{\alpha+\beta}u_t\nabla u
		+\hat{\nu}e^{2\psi}(u\nabla u+u^2\nabla\psi))\nonumber\\
	&&+e^{2\psi}\left[\frac{a_0}{4}-\hat{\nu}+(-\psi_t)\langle x\rangle_K^{\alpha+\beta}\right]u_t^2
		+e^{2\psi}\left[\hat{\nu}\delta_3+\frac{-\psi_t}{5}\langle x\rangle_K^{\alpha+\beta}\right]
		|\nabla u|^2\nonumber\\
	&&+e^{2\psi}\left[\hat{\nu}\left(\delta_3|\nabla\psi|^2+\frac{\delta}{3}(-\psi_t)a(x)b(t)
		+(B-2\delta_1)\frac{a(x)b(t)}{2(1+t)}\right)\right]u^2\nonumber\\
	&&+e^{2\psi}[\underbrace{(\alpha+\beta)\langle x\rangle_K^{\alpha+\beta-2}x\cdot u_t\nabla u
		-2\hat{\nu}\psi_tuu_t}_{T_4}]\nonumber\\
	&\le&\frac{\partial}{\partial t}\left[e^{2\psi}\langle x\rangle_K^{\alpha+\beta}F(u)\right]
		+2e^{2\psi}\langle x\rangle_K^{\alpha+\beta}(-\psi_t)F(u)
		+\hat{\nu}e^{2\psi}uf(u).\label{eq221}
\end{eqnarray}
The terms $T_4$ can be estimated as
$$
	|(\alpha+\beta)\langle x\rangle_K^{\alpha+\beta-2}x\cdot u_t\nabla u|
	\le\frac{\hat{\nu}\delta_3}{2}|\nabla u|^2
	+\frac{(\alpha+\beta)^2}{2\hat{\nu}\delta_3K^{2(1-\alpha-\beta)}}u_t^2,
$$
$$|2\hat{\nu}(-\psi_t)uu_t|\le
	\frac{\hat{\nu}\delta}{3}(-\psi_t)a(x)b(t)u^2
	+\frac{3\hat{\nu}}{a_0\delta}(-\psi_t)\langle x\rangle_K^{\alpha+\beta}u_t^2.
$$
From this we can rewrite (\ref{eq221}) as
\begin{eqnarray}
	\lefteqn{\frac{\partial}{\partial t}\left[e^{2\psi}\left(\frac{\langle x\rangle_K^{\alpha+\beta}}{2}u_t^2
		+\hat{\nu}uu_t+\frac{\hat{\nu}a(x)b(t)}{2}u^2+\frac{\langle x\rangle_K^{\alpha+\beta}}{2}|\nabla u|^2
		\right)\right] }\nonumber\\
	&&-\nabla\cdot(e^{2\psi}\langle x\rangle_K^{\alpha+\beta}u_t\nabla u
		+\hat{\nu}e^{2\psi}(u\nabla u+u^2\nabla\psi))\nonumber\\
	&&+e^{2\psi}\left[\left(\frac{a_0}{4}
		-\hat{\nu}-\frac{(\alpha+\beta)^2}{2\hat{\nu}\delta_3K^{2(1-\alpha-\beta)}}\right)
		+\left(1-\frac{3\hat{\nu}}{a_0\delta}\right)(-\psi_t)\langle x\rangle_K^{\alpha+\beta}\right]u_t^2\nonumber\\
	&&+e^{2\psi}\left[\frac{\hat{\nu}\delta_3}{2}
		+\frac{-\psi_t}{5}\langle x\rangle_K^{\alpha+\beta}\right]|\nabla u|^2\nonumber\\
	&&+e^{2\psi}
		\left[\hat{\nu}\left(\delta_3|\nabla\psi|^2+(B-2\delta_1)\frac{a(x)b(t)}{2(1+t)}\right)\right]u^2\nonumber\\
	&\le&\frac{\partial}{\partial t}\left[e^{2\psi}\langle x\rangle_K^{\alpha+\beta}F(u)\right]
		+2e^{2\psi}\langle x\rangle_K^{\alpha+\beta}(-\psi_t)F(u)
		+\hat{\nu}e^{2\psi}uf(u).\label{eq222}
\end{eqnarray}
Now we choose the parameters $\hat{\nu}$ and $K$ in the same manner as before.
Indeed taking $\hat{\nu}$ sufficiently small depending on $\delta$
and then choosing $K$ sufficiently large depending on $\hat{\nu}$, we can obtain
$$
	\frac{a_0}{4}-\hat{\nu}-\frac{(\alpha+\beta)^2}{2\hat{\nu}\delta_3K^{2(1-\alpha-\beta)}}\ge c_1,\quad
	1-\frac{3\hat{\nu}}{a_0\delta}\ge c_1,\quad
	\nu\delta_3\ge c_1,\quad
	\frac{1}{5}\ge c_1
$$
for some constant $c_1>0$.
Consequently, By integrating (\ref{eq222}) on $\Omega^c$,
the energy inequality on $\Omega^c$ follows:
\begin{equation}
\frac{d}{dt}\overline{E}_{\psi}(t;\Omega^c(t;K,t_0))+N_2(t)+M_2(t)+H_{\psi}(t;\Omega^c(t;K,t_0))\le P_2,
\label{eq223}
\end{equation}
where
\begin{eqnarray*}
	\overline{E}_{\psi}(t;\Omega^c)&=&\overline{E}_{\psi}(t;\Omega^c(t;K,t_0))\\
		&:=&\int_{\Omega^c}e^{2\psi}\left(\frac{\langle x\rangle_K^{\alpha+\beta}}{2}u_t^2
		+\hat{\nu} uu_t+\frac{\hat{\nu} a(x)b(t)}{2}u^2+
		\frac{\langle x\rangle_K^{\alpha+\beta}}{2}|\nabla u|^2\right)dx,
\end{eqnarray*}
\begin{eqnarray*}
	N_2(t)&:=&\int_{\mathbf{S}^{n-1}}e^{2\psi}
		\left(\frac{\langle x\rangle_K^{\alpha+\beta}}{2}u_t^2
		+\hat{\nu} uu_t+\frac{\hat{\nu} a(x)b(t)}{2}u^2\right.\\
	&&+\frac{\langle x\rangle_K^{\alpha+\beta}}{2}|\nabla u|^2)\Big|_{|x|=\sqrt{(t_0+t)^2-K}}\\
	&&\times [(t_0+t)^2-K]^{(n-1)/2}d\theta\cdot\frac{d}{dt}\sqrt{(t_0+t)^2-K},
\end{eqnarray*}
$$
	M_2(t):=\int_{\partial\Omega^c}(e^{2\psi}\langle x\rangle_K^{\alpha+\beta}u_t\nabla u
	+\hat{\nu} e^{2\psi}(u\nabla u+u^2\nabla\psi))\cdot\vec{n}dS,
$$
\begin{eqnarray*}
	H_{\psi}(t;\Omega^c)&=&H_{\psi}(t;\Omega^c(t;K,t_0))\\
	&:=&c_1\int_{\Omega}e^{2\psi}(1+\langle x\rangle_K^{\alpha+\beta}(-\psi_t))(u_t^2+|\nabla u|^2)dx\\
	&&+\hat{\nu} (B-2\delta_1)\int_{\Omega^c}\frac{e^{2\psi}a(x)b(t)}{2(1+t)}u^2dx,
\end{eqnarray*}
\begin{eqnarray*}
	P_2&:=&\frac{d}{dt}\left[\int_{\Omega^c}e^{2\psi}\langle x\rangle_K^{\alpha+\beta}F(u)dx\right]\\
		&&+\int_{\mathbf{S}^{n-1}}\left.\langle x\rangle_K^{\alpha+\beta}
			e^{2\psi}F(u)\right|_{|x|=\sqrt{(t_0+t)^2-K}}\\
		&&\times [(t_0+t)^2-K]^{(n-1)/2}d\theta\cdot\frac{d}{dt}\sqrt{(t_0+t)^2-K}\\
		&&+C\int_{\Omega^c}e^{2\psi}(1+\langle x\rangle_K^{\alpha+\beta}(-\psi_t))|u|^{p+1}dx.
\end{eqnarray*}
In a similar way as the case in $\Omega$, we note that
\begin{eqnarray*}
\lefteqn{c\int_{\Omega^c}e^{2\psi}(t_0+t)^{\alpha+\beta}(u_t^2+|\nabla u|^2)dx
	+c\int_{\Omega^c}e^{2\psi}a(x)b(t)u^2dx}\\
&\le&\overline{E}_{\psi}(t;\Omega^c(t;K,t_0))\\
&\le&C\int_{\Omega^c}e^{2\psi}(t_0+t)^{\alpha+\beta}(u_t^2+|\nabla u|^2)dx
	+C\int_{\Omega^c}e^{2\psi}a(x)b(t)u^2dx
\end{eqnarray*}
for some constants $c>0$ and $C>0$.

We add the energy inequalities on $\Omega$ and $\Omega^c$.
We note that replacing $\nu$ and $\hat{\nu}$ by $\nu_0:=\min\{\nu, \hat{\nu}\}$,
we can still have the inequalities (\ref{eq219}) and (\ref{eq223}),
provided that we retake $t_0$ and $K$ larger.

By$((\ref{eq219})+(\ref{eq223}))\times (t_0+t)^{B-\varepsilon}$, we have
\begin{eqnarray}
	\lefteqn{\frac{d}{dt}[(t_0+t)^{B-\varepsilon}
		(\overline{E}_{\psi}(t;\Omega)+\overline{E}_{\psi}(t;\Omega^c))]}\nonumber\\
	&-&\underbrace{(B-\varepsilon)(t_0+t)^{B-1-\varepsilon}
		(\overline{E}_{\psi}(t;\Omega)+\overline{E}_{\psi}(t;\Omega^c))}_{T_5}\nonumber\\
	&+&\underbrace{(t_0+t)^{B-\varepsilon}(H_{\psi}(t;\Omega)+H_{\psi}(t;\Omega^c))}_{T_6}\nonumber\\
	&\le&(t_0+t)^{B-\varepsilon}(P_1+P_2)\label{eq224},
\end{eqnarray}
here we note that
$$N_1(t)=N_2(t),\quad M_1(t)=M_2(t)$$
on $\partial\Omega$.
Since
$$|\nu_0 uu_t|\le\frac{\nu_0\delta_4}{2}a(x)b(t)u^2+\frac{\nu_0}{2\delta_4a_0}(t_0+t)^{\alpha+\beta}u_t^2$$
on $\Omega$ and
$$|\nu_0uu_t|\le\frac{\nu_0\delta_4}{2}a(x)b(t)u^2+\frac{\nu_0}{2\delta_4a_0}
	\langle x\rangle_K^{\alpha+\beta}u_t^2$$
on $\Omega^c$, we have
\begin{equation}
\label{eq225}
-T_5+T_6\ge (t_0+t)^{B-\varepsilon}I_1+(t_0+t)^{B-\varepsilon}I_2,
\end{equation}
where
\begin{eqnarray*}
I_1&:=&\int_{\Omega}
			e^{2\psi}\left\{\frac{c_0}{2}(1+(t_0+t)^{\alpha+\beta}(-\psi_t))\right.
		\left. -\frac{B-\varepsilon}{2(t_0+t)}
			\left(1+\frac{2\nu_0}{\delta_4a_0}\right)(t_0+t)^{\alpha+\beta}\right\}u_t^2\\
	&&+e^{2\psi}\left\{\frac{c_0}{2}(1+(t_0+t)^{\alpha+\beta}(-\psi_t))
		-\frac{B-\varepsilon}{2(t_0+t)}(t_0+t)^{\alpha+\beta}\right\}|\nabla u|^2dx\\
	&&+\int_{\Omega^c}e^{2\psi}\left\{\frac{c_1}{2}(1+\langle x\rangle_K^{\alpha+\beta}(-\psi_t))
		-\frac{B-\varepsilon}{2(t_0+t)}\left(1+\frac{2\nu_0}{\delta_4a_0}\right)
				\langle x\rangle_K^{\alpha+\beta}\right\}u_t^2\\
	&&+e^{2\psi}\left\{\frac{c_1}{2}(1+\langle x\rangle_K^{\alpha+\beta}(-\psi_t))
			-\frac{B-\varepsilon}{2(t_0+t)}
			\langle x\rangle_K^{\alpha+\beta}\right\}|\nabla u|^2dx\\
		&&=:I_{11}+I_{12},\\
I_2&:=&\nu_0(B-2\delta_1-(1+\delta_4)(B-\varepsilon))
		\left(\int_{\Omega}+\int_{\Omega^c}\right)
			e^{2\psi}\frac{a(x)b(t)}{2(1+t)}u^2dx\\
	&&+\frac{c_2}{2}\int_{\mathbf{R}^n}e^{2\psi}(u_t^2+|\nabla u|^2)dx,
\end{eqnarray*}
where $c_2:=\min(c_0, c_1)$.
Recall the definition of $\varepsilon$ and $\delta_1$ (i.e. (\ref{eqepsilon}) and (\ref {eq26})).
A simple calculation shows $\varepsilon =3\delta_1$.
Choosing $\delta_4$ sufficiently small depending on $\varepsilon$,
we have
$$(t_0+t)^{B-\varepsilon}I_2\ge
	c_3(t_0+t)^{B-1-\varepsilon}\int_{\mathbf{R}^n}e^{2\psi}a(x)b(t)u^2dx
	+\frac{c_2}{2}(t_0+t)^{B-\varepsilon}E(t)$$
for some constant $c_3>0$.
Next, we prove that $I_1\ge 0$.
By noting that $\alpha+\beta<1$, it is easy to see that $I_{11}\ge 0$ if we retake $t_0$
larger depending on $c_0, \nu_0$ and $\delta_4$.
To estimate $I_{12}$, we further divide the region $\Omega^c$ into
$$\Omega^c(t;K,t_0)=(\Omega^c(t;K,t_0)\cap\Sigma_L)\cup(\Omega^c(t;K,t_0)\cap\Sigma_L^c),$$
where
$$\Sigma_L:=\{x\in\mathbf{R}^n; \langle x\rangle^{2-\alpha}\le L(1+t)^{1+\beta}\},
	\quad\Sigma_L^c:=\mathbf{R}^n\setminus\Sigma_L$$
with $L\gg 1$ determined later.
First, since $K+|x|^2\le K(1+|x|^2)\le KL^{2/(2-\alpha)}(1+t)^{2(1+\beta)/(2-\alpha)}$
on $\Omega^c\cap\Sigma_L$, we have
\begin{eqnarray*}
\lefteqn{\frac{c_1}{2}(1+\langle x\rangle_K^{\alpha+\beta}(-\psi_t))
	-\frac{B-\varepsilon}{2(t_0+t)}
		\left(1+\frac{2\nu_0}{\delta_4a_0}\right)\langle x\rangle_K^{\alpha+\beta}}\\
	&\ge&\frac{c_1}{2}-\frac{B-\varepsilon}{2(t_0+t)}\left(1+\frac{2\nu_0}{\delta_4a_0}\right)
	K^{(\alpha+\beta)/2}L^{(\alpha+\beta)/(2-\alpha)}(1+t)^{\frac{(1+\beta)(\alpha+\beta)}{2-\alpha}}.
\end{eqnarray*}
We note that $-1+\frac{(1+\beta)(\alpha+\beta)}{2-\alpha}<0$ by $\alpha+\beta<1$.
Thus, we obtain
$$\frac{c_1}{2}-\frac{B-\varepsilon}{2(t_0+t)}\left(1+\frac{2\nu_0}{\delta_4a_0}\right)
	K^{(\alpha+\beta)/2}L^{(\alpha+\beta)/(2-\alpha)}
	(1+t)^{\frac{(1+\beta)(\alpha+\beta)}{2-\alpha}}\ge 0$$
for large $t_0$ depending on $L$ and $K$.
Secondly, on $\Omega^c\cap\Sigma_L^c$, we have
\begin{eqnarray*}
\lefteqn{\frac{c_1}{2}(1+\langle x\rangle_K^{\alpha+\beta}(-\psi_t))
	-\frac{B-\varepsilon}{2(t_0+t)}\left(1+\frac{2\nu_0}{\delta_4a_0}\right)\langle x\rangle_K^{\alpha+\beta}}\\
&\ge&\left\{\frac{c_1}{2}(1+\beta)\frac{\langle x\rangle^{2-\alpha}}{(1+t)^{2+\beta}}
	-\frac{B-\varepsilon}{2(t_0+t)}\left(1+\frac{2\nu_0}{\delta_4a_0}\right)\right\}\langle x\rangle_K^{\alpha+\beta}\\
&\ge&\left\{\frac{c_1}{2}(1+\beta)\frac{L}{1+t}
	-\frac{B-\varepsilon}{2(t_0+t)}\left(1+\frac{2\nu_0}{\delta_4a_0}\right)\right\}\langle x\rangle_K^{\alpha+\beta}.
\end{eqnarray*}
Therefore one can obtain $I_{12}\ge 0$,
provided that $L\ge \frac{B-\varepsilon}{c_1(1+\beta)}(1+\frac{2\nu_0}{\delta_4a_0})$.
Consequently, we have $I_1\ge 0$.
By (\ref{eq225}) and that we mentioned above, it follows that
\begin{equation*}
-T_5+T_6\ge c_3(t_0+t)^{B-1-\varepsilon}\int_{\mathbf{R}^n}e^{2\psi}a(x)b(t)u^2dx
		+\frac{c_2}{2}(t_0+t)^{B-\varepsilon}E(t).
\end{equation*}
Therefore, we have
\begin{eqnarray}
	\lefteqn{\frac{d}{dt}[(t_0+t)^{B-\varepsilon}(\overline{E}_{\psi}(t;\Omega)+\overline{E}_{\psi}(t;\Omega^c)]
		+\frac{c_2}{2}(t_0+t)^{B-\varepsilon}E(t)}\nonumber\\
		&&+c_3(t_0+t)^{B-1-\varepsilon}J(t;a(x)b(t)u^2)\nonumber\\
		&&\le(t_0+t)^{B-\varepsilon}(P_1+P_2)\label{eq226}.
\end{eqnarray}
Integrating (\ref{eq226}) on the interval $[0,t]$, one can obtain the energy inequality on the whole space:
\begin{eqnarray}
	\lefteqn{(t_0+t)^{B-\varepsilon}(\overline{E}_{\psi}(t;\Omega)+\overline{E}_{\psi}(t;\Omega^c))
		+\frac{c_2}{2}\int_0^t(t_0+\tau)^{B-\varepsilon}E(\tau)d\tau}\nonumber\\
	&&+c_3\int_0^t(t_0+\tau)^{B-1-\varepsilon}J(\tau;a(x)b(\tau)u^2)d\tau\nonumber\\
	&\le&CI_0^2+\int_0^t(t_0+\tau)^{B-\varepsilon}(P_1+P_2)d\tau\label{eq227}.
\end{eqnarray}
By $(\ref{eq227})+\mu\times (\ref{eq212})$, here $\mu$ is a small positive parameter determined later, it follows that
\begin{eqnarray}
	\lefteqn{(t_0+t)^{B-\varepsilon}\overline{E}_{\psi}(t;\Omega)
		+(t_0+t)^{B-\varepsilon}\overline{E}_{\psi}(t;\Omega^c)}\nonumber\\
	&&+\int_0^t\frac{c_2}{2}(t_0+\tau)^{B-\varepsilon}E(\tau)-\mu C(t_0+\tau)^{B-\varepsilon}E(\tau)d\tau\nonumber\\
	&&+c_3\int_0^t(t_0+\tau)^{B-1-\varepsilon}J(\tau;a(x)b(\tau)u^2)d\tau
		+\mu(t_0+t)^{B+1-\varepsilon}E(t)\nonumber\\
	&&+\mu\int_0^t(t_0+\tau)^{B+1-\varepsilon}J(\tau;a(x)b(\tau)u_t^2)
		+(t_0+\tau)^{B+1-\varepsilon}E_{\psi}(\tau)d\tau\nonumber\\
	&\le&CI_0^2+P\nonumber\\
	&&+C(t_0+t)^{B+1-\varepsilon}J(t;|u|^{p+1})\nonumber\\
		&&+C\int_0^t(t_0+\tau)^{B+1-\varepsilon}J_{\psi}(\tau;|u|^{p+1})d\tau\nonumber\\
	&&+C\int_0^t(t_0+\tau)^{B-\varepsilon}J(\tau;|u|^{p+1})d\tau,\label{eq228}
\end{eqnarray}
where
$$
	P=\int_0^t(t_0+\tau)^{B-\varepsilon}(P_1+P_2)d\tau.
$$
Now we choose $\mu$ sufficiently small, then we can rewrite (\ref{eq228}) as
\begin{eqnarray}
	\lefteqn{(t_0+t)^{B+1-\varepsilon}E(t)+(t_0+t)^{B-\varepsilon}J(t;a(x)b(t)u^2)}\nonumber\\
	&\le&CI_0^2+P+C(t_0+t)^{B+1-\varepsilon}J(t;|u|^{p+1})\nonumber\\
	&&+C\int_0^t(t_0+\tau)^{B+1-\varepsilon}J_{\psi}(\tau;|u|^{p+1})d\tau\nonumber\\
	&&+C\int_0^t(t_0+\tau)^{B-\varepsilon}J(\tau;|u|^{p+1})d\tau.\label{eq229}
\end{eqnarray}
We shall estimate the right hand side of (\ref{eq229}).
We need the following lemma.
\begin{lemma}[Gagliardo-Nirenberg]\label{lem22}
Let $p,q,r (1\le p,q,r\le\infty)$ and $\sigma\in [0,1]$ satisfy
$$\frac{1}{p}=\sigma\left(\frac{1}{r}-\frac{1}{n}\right)+(1-\sigma)\frac{1}{q}$$
except for $p=\infty$ or $r=n$ when $n\ge 2$. Then for some constant $C=C(p,q,r,n)>0$,
the inequality
$$\|h\|_{L^p}\le C\|h\|_{L^q}^{1-\sigma}\|\nabla h\|_{L^r}^{\sigma},
		\qquad \mbox{for any} \quad h\in C^1_0(\mathbf{R}^n)$$
holds.
\end{lemma}
We first estimate $(t_0+t)^{B+1-\varepsilon}J(t;|u|^{p+1})$. From the above lemma, we have
\begin{eqnarray}
J(t;|u|^{p+1})&\le&C\left(\int_{\mathbf{R}^n}e^{\frac{4}{p+1}\psi}u^2dx\right)^{(1-\sigma)(p+1)/2}\nonumber\\
&&\times\left(\int_{\mathbf{R}^n}e^{\frac{4}{p+1}\psi}|\nabla\psi|^2u^2dx\right.\nonumber\\
		&&\qquad\qquad\left.+\int_{\mathbf{R}^n}e^{\frac{4}{p+1}\psi}|\nabla u|^2dx
				\right)^{\sigma(p+1)/2}\label{eq230}
\end{eqnarray}
with $\sigma=\frac{n(p-1)}{2(p+1)}$.
Since
\begin{eqnarray*}
e^{\frac{4}{p+1}\psi}u^2&=&(e^{2\psi}a(x)b(t)u^2)a(x)^{-1}b(t)^{-1}e^{\left(\frac{4}{p+1}-2\right)\psi}\\
&\le&C(e^{2\psi}a(x)b(t)u^2)\left[
		\left(\frac{\langle x\rangle^{2-\alpha}}{(1+t)^{1+\beta}}\right)^{\frac{\alpha}{2-\alpha}}
		e^{\left(\frac{4}{p+1}-2\right)\psi}\right]\\
		&&\qquad\qquad\times(1+t)^{\beta+(1+\beta)\alpha/(2-\alpha)}\\
&\le&C(1+t)^{\beta+(1+\beta)\alpha/(2-\alpha)}e^{2\psi}a(x)b(t)u^2
\end{eqnarray*}
and
\begin{eqnarray*}
e^{\frac{4}{(p+1)}\psi}|\nabla\psi|^2u^2&\le&
	C\frac{\langle x\rangle^{2-2\alpha}}{(1+t)^{2+2\beta}}
	e^{\frac{1}{2}\left(\frac{4}{p+1}-2\right)\psi}e^{\frac{1}{2}\left(\frac{4}{p+1}-2\right)\psi}e^{2\psi}u^2\\
&\le&Ce^{\frac{1}{2}\left(\frac{4}{p+1}-2\right)\psi}e^{2\psi}\left[
		\left(\frac{\langle x\rangle^{2-\alpha}}{(1+t)^{1+\beta}}\right)^{\frac{2-2\alpha}{2-\alpha}}
		e^{\frac{1}{2}\left(\frac{4}{p+1}-2\right)\psi}\right]\\
&&\times (1+t)^{-2(1+\beta)+(1+\beta)(2-2\alpha)/(2-\alpha)}u^2\\
&\le&C(1+t)^{-2(1+\beta)/(2-\alpha)}e^{\frac{1}{2}\left(\frac{4}{p+1}-2\right)\psi}e^{2\psi}u^2\\
&\le&C(1+t)^{-2(1+\beta)/(2-\alpha)}(1+t)^{\beta+(1+\beta)\alpha/(2-\alpha)}e^{2\psi}a(x)b(t)u^2,
\end{eqnarray*}
we can estimate (\ref{eq230}) as
\begin{eqnarray*}
J(t;|u|^{p+1})&\le&
C(1+t)^{[\beta+(1+\beta)\alpha/(2-\alpha)](1-\sigma)(p+1)/2}J(t;a(x)b(t)u^2)^{(1-\sigma)(p+1)/2}\nonumber\\
&&\times[(1+t)^{-1}
		J(t;a(x)b(t)u^2)+E(t)]^{\sigma(p+1)/2}
\end{eqnarray*}
and hence
\begin{equation*}
(t_0+t)^{B+1-\varepsilon}J(t;|u|^{p+1})
\le C\left((t_0+t)^{\gamma_1}M(t)^{(p+1)/2}+(t_0+t)^{\gamma_2}M(t)^{(p+1)/2}\right)\label{eq232},
\end{equation*}
where
\begin{eqnarray*}
	\gamma_1&=&B+1-\varepsilon+\left[\beta+(1+\beta)\frac{\alpha}{2-\alpha}\right]\frac{1-\sigma}{2}(p+1)
		-\frac{\sigma}{2}(p+1)\nonumber\\
		&&-(B-\varepsilon)\frac{p+1}{2},\nonumber\\
	\gamma_2&=&B+1-\varepsilon+\left[\beta+(1+\beta)\frac{\alpha}{2-\alpha}\right]\frac{1-\sigma}{2}(p+1)
		-(B-\varepsilon)\frac{1-\sigma}{2}(p+1)\nonumber\\
		&&-(B+1-\varepsilon)\frac{\sigma}{2}(p+1).
\end{eqnarray*}
By a simple calculation it follows that if
$$p>1+\frac{2}{n-\alpha},$$
then by taking $\varepsilon$ sufficiently small (i.e. $\delta$ sufficiently small)
both $\gamma_1$ and $\gamma_2$ are negative.
We note that
\begin{eqnarray*}
J_{\psi}(t;|u|^{p+1})&=&\int_{\mathbf{R}^n}e^{2\psi}(-\psi_t)|u|^{p+1}dx\\
&\le&\frac{C}{1+t}\int_{\mathbf{R}^n}e^{(2+\rho)\psi}|u|^{p+1}dx,
\end{eqnarray*}
where $\rho$ is a sufficiently small positive number.
Therefore, we can estimate the terms
$$
	\int_0^t(t_0+\tau)^{B+1-\varepsilon}J_{\psi}(\tau;|u|^{p+1})d\tau\quad
		\mbox{and}\quad \int_0^t(t_0+\tau)^{B-\varepsilon}J(\tau;|u|^{p+1})d\tau
$$
in the same manner as before.
Noting that
\begin{eqnarray*}
P_1+P_2&=&
\frac{d}{dt}\left[(t_0+t)^{\alpha+\beta}\int_{\Omega}e^{2\psi}F(u)dx
	+\int_{\Omega^c}e^{2\psi}\langle x\rangle_K^{\alpha+\beta}F(u)dx\right]\\
&&+C\int_{\Omega}e^{2\psi}(1+(t_0+t)^{\alpha+\beta}(-\psi_t))|u|^{p+1}dx\\
&&+C\int_{\Omega^c}e^{2\psi}(1+\langle x\rangle_K^{\alpha+\beta}(-\psi_t))|u|^{p+1}dx,
\end{eqnarray*}
we have
\begin{eqnarray*}
P&=&\int_0^t(t_0+\tau)^{B-\varepsilon}(P_1+P_2)d\tau\\
&\le&CI_0^2
	+C(t_0+t)^{B-\varepsilon}\int_{\Omega}e^{2\psi}(t_0+t)^{\alpha+\beta}F(u)dx\nonumber\\
	&&+C(t_0+t)^{B-\varepsilon}\int_{\Omega^c}e^{2\psi}\langle x\rangle_K^{\alpha+\beta}F(u)dx\nonumber\\
	&&+C\int_0^t(t_0+\tau)^{B-1-\varepsilon}
		\int_{\Omega}e^{2\psi}(t_0+\tau)^{\alpha+\beta}F(u)dxd\tau\nonumber\\
	&&+C\int_0^t(t_0+\tau)^{B-1-\varepsilon}
		\int_{\Omega^c}e^{2\psi}\langle x\rangle_K^{\alpha+\beta}F(u)dxd\tau\nonumber\\
	&&+C\int_0^t(t_0+\tau)^{B-\varepsilon}\int_{\Omega}e^{2\psi}
		(1+(t_0+\tau)^{\alpha+\beta}(-\psi_t))|u|^{p+1}dxd\tau\nonumber\\
	&&+C\int_0^t(t_0+\tau)^{B-\varepsilon}
		\int_{\Omega^c}e^{2\psi}(1+\langle x\rangle_K^{\alpha+\beta}(-\psi_t))|u|^{p+1}dxd\tau.\nonumber
\end{eqnarray*}
We calculate
\begin{eqnarray*}
e^{2\psi}\langle x\rangle_K^{\alpha+\beta}&=&
e^{2A\frac{\langle x\rangle^{2-\alpha}}{(1+t)^{1+\beta}}}\langle x\rangle_K^{\alpha+\beta}\\
&\le&Ce^{2A\frac{\langle x\rangle^{2-\alpha}}{(1+t)^{1+\beta}}}
	\left(\frac{\langle x\rangle^{2-\alpha}}{(1+t)^{1+\beta}}\right)^{\frac{\alpha+\beta}{2-\alpha}}
		(1+t)^{\frac{(\alpha+\beta)(1+\beta)}{2-\alpha}}\\
&\le&Ce^{(2+\rho)\psi}(1+t)^{\frac{(\alpha+\beta)(1+\beta)}{2-\alpha}}
\end{eqnarray*}
for small $\rho >0$.
Noting that $\frac{(\alpha+\beta)(1+\beta)}{2-\alpha}<1$ and taking $\rho$ sufficiently small,
we can estimate the terms $P$ in the same manner as estimating $(t_0+t)^{B+1-\varepsilon}J(t;|u|^{p+1})$.
Consequently, we have a priori estimate for $M(t)$:
$$M(t)\le CI_0^2+CM(t)^{(p+1)/2}.$$
This shows that the local solution of (\ref{eq11}) can be extended globally.
We note that
$$
	e^{2\psi}a(x)b(t)\ge c(1+t)^{-(1+\beta)\frac{\alpha}{2-\alpha}-\beta}
$$
with some constant $c>0$.
Then we have
\begin{equation}
\label{eq231}
	\int_{\mathbf{R}^n}e^{2\psi}a(x)b(t)u^2dx
	\ge c(1+t)^{-(1+\beta)\frac{\alpha}{2-\alpha}-\beta}\int_{\mathbf{R}^n}u^2dx.
\end{equation}
This implies the decay estimate of global solution (\ref{eq18}) and completes 
the proof of Theorem \ref{thm1}.

\begin{proof}[Proof of Corollary \ref{cor13}]
In a similar way to derive (\ref{eq231}), we have
$$\int_{\mathbf{R}^n}e^{2\psi}a(x)b(t)u^2dx
	\ge c(1+t)^{-(1+\beta)\frac{\alpha}{2-\alpha}-\beta}\int_{\mathbf{R}^n}
	e^{(2A-\mu)\frac{\langle x\rangle^{2-\alpha}}{(1+t)^{\beta}}}u^2dx.$$
By noting that
$$\frac{\langle x\rangle^{2-\alpha}}{(1+t)^{1+\beta}}\ge (1+t)^{\rho}$$
on $\Omega_{\rho}(t)$ and Theorem \ref{thm1}, it follows that
\begin{eqnarray*}
\lefteqn{(1+t)^{-(1+\beta)\frac{\alpha}{2-\alpha}-\beta}\int_{\Omega_{\rho}(t)}
		e^{(2A-\mu)(1+t)^{\rho}}(u_t^2+|\nabla u|^2+u^2)dx}\\
&\le&C(1+t)^{-(1+\beta)\frac{\alpha}{2-\alpha}-\beta}\int_{\Omega_{\rho}(t)}
		e^{(2A-\mu)\frac{\langle x\rangle^{2-\alpha}}{(1+t)^{\beta}}}(u_t^2+|\nabla u|^2+u^2)dx\\
&\le&C\int_{\mathbf{R}^n}e^{2\psi}(u_t^2+|\nabla u|^2+a(x)b(t)u^2)dx\\
&\le&C(1+t)^{-B+\varepsilon}.
\end{eqnarray*}
Thus, we obtain
$$\int_{\Omega_{\rho}(t)}(u_t^2+|\nabla u|^2+u^2)dx\le C(1+t)^{-\frac{(1+\beta)(n-2\alpha)}{2-\alpha}+\varepsilon}
		e^{-(2A-\mu)(1+t)^{\rho}}.$$
This proves Corollary \ref{cor13}.
\end{proof}

\begin{center}
Acknowledgement
\end{center}
The author is deeply grateful to Professors Ryo Ikehata, Kenji Nishihara, Tatsuo Nishitani, Akitaka Matsumura
and Michael Reissig.
They gave me constructive comments and warm encouragement again and again.

\end{document}